\documentclass[12pt]{article}
\usepackage[margin=1in]{geometry} 
\usepackage{amsmath,amsthm,amssymb}
\usepackage[english]{babel}
\usepackage[utf8]{inputenc} 
\usepackage{algorithm}
\usepackage{algorithmic}
\usepackage{graphicx}
\usepackage{color,hyperref}
\graphicspath{ {./images/} }
\usepackage{amsmath,arydshln}

\newcommand{\sF}{{\mathcal F}}

\newcommand{\F}{{\mathbb{F}}}

\newcommand{\N}{{\mathbb N}}

\newcommand{\R}{{\mathbb R}}

\usepackage{setspace}

\newcommand{\be}{\begin{eqnarray}}
\newcommand{\ee}{\end{eqnarray}}
\newcommand{\nn}{{\nonumber}}
\newcommand{\dd}{\displaystyle}

\newcommand{\Tr}{\mbox{Tr}}
\newcommand{\Norm}{\mbox{N}}

\newtheorem*{theoremA}{Theorem A}

\theoremstyle{definition}
\newtheorem{theorem}{Theorem}
\newtheorem{definition}{Definition}
\newtheorem{corollary}{Corollary}
\newtheorem{lemma}{Lemma}

\newtheorem{example}{Example}
\newtheorem{construction}{Construction}
\newtheorem{proposition}{Proposition}

\newenvironment{proofA}{\noindent\textit{Proof of Theorem 1}}{\hfill{$\Box$}}

\theoremstyle{remark} 

\providecommand{\keywords}[1]
{
	\small	
	\textbf{{Keywords:}} #1
}
\providecommand{\amssub}[1]
{
	\small	
	\textbf{{Subject Classification:}} #1
}
\catcode`\_=11\relax
\newcommand\email[1]{\_email #1\q_nil}
\def\_email#1@#2\q_nil{%
	\href{mailto:#1@#2}{{\emailfont #1\emailampersat #2}}
}
\newcommand\emailfont{\sffamily}
\newcommand\emailampersat{{\color{black}\small@}}
\catcode`\_=8\relax    
\begin{document}
\title{Improved lower bound on the family complexity of Legendre sequences}
\author{Yağmur Çakıroğlu and Oğuz Yayla \\
	\small Department of Mathematics, Hacettepe University\\
	\small Beytepe, 06800, Ankara, Turkey\\
	\small\email{yagmur.cakiroglu@hacettepe.edu.tr}\\
	\small\email{oguz.yayla@hacettepe.edu.tr}
}

\maketitle

\begin{abstract}In this paper we study a family of binary Legendre sequences and its family complexity. Family complexity is a pseudorandomness measure introduced by Ahlswede et al. in 2003. A lower bound on the family complexity of a family based on the Legendre symbol of polynomials over a finite field was given by Gyarmati in 2015. In this article we improve bound given by Gyarmati. The new bound depends on the Lambert $W$ function and the number of elements in a finite field belonging to its proper subfield.


\amssub{11K45 \and 94A55 \and  94A60}

\end{abstract}

\section{Introduction}

A pseudorandom sequence is a sequence of numbers which is generated by a deterministic algorithm and looks truly random. A pseudorandom sequence in the interval [0, 1) is called a sequence of pseudorandom numbers. Randomness measures of a sequence depend on its application area, for instance, it has to be unpredictable for cryptographic applications, uncorrelated for wireless communication applications and uniformly distributed for quasi-Monte Carlo methods \cite{golomb2005signal,NA2015}.

In this paper we consider Legendre sequences.
It is known that the Legendre sequence has several good randomness measures such as high linear complexity \cite{aly2006,ding1998linear,Spa2013,turyn1964linear} and small correlation measure up to rather high orders \cite{AMS1997} for cryptography, and a small (aperiodic) autocorrelation \cite{MW2004,paley1933orthogonal} for wireless communication, GPS, radar or sonar.



In case, a family of sequences is considered for an application, for instance as a key-space of a cryptosystem, then its randomness in terms of many directions is concerned. For instance, a family of sequences must have large family size, large family complexity and low crosscorrelation.
Family complexity as a randomness measure was first introduced by Ahlswede, Khachatrian, Mauduit and Sárközy \cite{AKMS2003} in 2003, and they estimated the family complexity of some sequences. 
Then, in 2006, they studied families of pseudorandom sequences on $k$-symbols and their family complexity \cite{AMS20061,AMS20062}. 
In 2013 Mauduit and Sárközy studied family complexity measure of sequences of $k$ symbols and they also gave the connection between family complexity and VC-dimension \cite{AMS2013}.
In 2016, Winterhof and the second author gave a relation between family complexity and cross correlation measure \cite{WO2016}. Moreover the complexity measures for different families have been studied in some papers \cite{BRD2014,F2010,GMS2004,GMS2012,GMS2014,S2017}.

Recently Gyarmati \cite{G2015} presented a bound for the family complexity of Legendre sequences. In this paper we improve Gyarmati's bound for all primes $p$ and degrees $k$. For instance the bound given in this article is positive for all $p$ but the bound is positive for $p\ge2128240847$. We plotted two bounds on family complexity of Legendre sequences for all primes $p$ and degrees $k$. We compare also two bounds in terms of time complexity. Then we plotted the difference between elapsed times for calculating bounds given by Gyarmati \cite{G2015} and Theorem 1. We give more details about these comparisons in the last section.

The paper is organized as follows. The new bound we present in this paper depends on Lambert $W$ function, so we give its definition and some properties in Section \ref{sec:LWF}. Then we present some auxiliary lemmas in Section \ref{sec:pre} and previous results in Section \ref{sec:previous}. Next, we give our main contribution in Section \ref{sec:contribution}. Finally we compare the new bound and Gyarmati's one in Section \ref{sec:compare}.

\section{Lambert $W$ Function}
\label{sec:LWF}
Firstly we begin with the definition of Lambert $W$ function and we present some properties and examples for this function.

\begin{definition} \label{def:W} \textbf{(Lambert $W$ Function)} The Lambert $W$ function, also called the omega function or product logarithm, is defined as the multivalued function $W$ that satisfies 
$$z=W(z)e^{W(z)}$$
for any complex number $z$. 
\end{definition}
Equivalently, Lambert $W$ function is known as the inverse function of $f(z)=ze^{z}$. 
Note that the multivaluedness of the Lambert $W$ function means that mostly there are multiple solutions since the function $f$ is not injective.  The equation $y=ze^{z}$ is by definition solved by $z=W(y)$ and the equation $y=z\log z$ is solved by $z=\frac{y}{W(y)}$. So many equations containing exponential expressions can be solved by the Lambert $W$ Function. For instance, the equation $xa^{x}=b$ is solved by  $$x=\frac{W(b\ln{(a)})}{\ln{(a)}},$$ the equation $a^{x}=x+b$ is solved by $$x=\frac{-b-W(-a^{-b}\ln{(a)})}{\ln{(a)}},$$ and the equation $x^{x^{a}}=b$ is solved by $$\exp{\bigg(\frac{W(a\log{(b)})}{a}\bigg)}.$$ 

The Lambert $W$ function has many applications in pure and applied mathematics, see \cite{LWF1996} for details about its applications. Lambert $W$ function stems form the equation proposed by Johann Heinrich Lambert in 1758 $$x^{\alpha}-x^{\beta}=(\alpha-\beta)vx^{\alpha+\beta},$$ which is known as Lambert's transcendental equation. Then in 1779 Euler wrote a paper about Lambert's transcendental equation and introduced special case which is nearly the definition of $W$ function \cite{E1779}. Actually Euler investigated theory behind the $W$ function and he had referenced work by Lambert in his paper, and so this function is called Lambert $W$ function. 
Lambert $W$ function, which has applications in many fields from past to present, was applied to problems ranging from quantum physics to population dynamics, to the complexity of algorithms.
The new bound we obtain for $f$-complexity given in this paper is related to this function. Now we give a simple example in order to show how we use this function in numerical solutions.
\begin{example} Let us solve  $4^t=3t$ for $t$. We first divide both sides by $4^{t}$  to get $$1 = 3t4^{-t} = 3t e^{-t\ln{4}}$$ 
and equivalently by multiplying $-\frac{\ln{(4)}}{3}$ we have
$$-\frac{\ln{(4)}}{3}=-t\ln{4}e^{-t\ln{4}}.$$
Since the right hand side of the equation is of the form $ze^{z}$ for $z = -t\ln t$, we can write the solution by definition of Lambert $W$ function 
$$t=\frac{W\bigg(-\frac{\ln{(4)}}{3}\bigg)}{-\ln{(4)}}$$ which is approximately $t \approx 0.611132623758349 - 0.480987054240275i$.
\end{example}
We note that $W$ function can be approximately evaluated by using some root-finding methods as given in \cite{LWF1996}.

\section{Preliminaries} \label{sec:pre}
In this section we present some definitions and results which we need for the proof of our new bound on family complexity.
\begin{definition}Let $\F_{q^n}$ denote the finite field having $q^n$ elements and define $G_{q,n}$ as follows.
$$G_{q,n}=\{\alpha \in \F_{q^n} : \exists t|n, t < n \mbox{ such that } \alpha\in \F_{q^t} \subset \F_{q^n}\}$$
\end{definition}

One can calculate the number of elements in $G_{q,n}$ for arbitrary $q,n$ by counting. But this method would be very slow. Thus, we need a formula for $|G_{q,n}|$, in order to do that we give some definitions and results below.
\begin{definition}\cite[Definition 2.1.22]{HBFF2013} The Möbius $\mu$ function is defined on the set of positive integers by
		\begin{displaymath}
	    \mu(m)= \left\{ \begin{array}{ll}
		1 & \textrm{if $m=1$}\\
		(-1)^{k} & \textrm{if $m=m_1m_2\dots m_k$ where the $m_i$ are distinct primes}\\
		0 & \textrm{if $p^2$ divides $m$ for some prime $p$}
		\end{array}\right.
		\end{displaymath} 
\end{definition}

Denote the number of monic irreducible polynomials of degree n over $\F_q$ by $I_q(n)$. Then the following formula is well known, see \cite[Chapter 14.3]{DFAA2004} or \cite{SJ2011} for its proof.

\begin{proposition} \cite{HBFF2013} (Gauss's Formula) For all $n\ge1$ any prime power q, we have 
$$I_{q}(n)=\frac{1}{n}\sum_{d|n}\mu(d)q^{n/d}$$
\end{proposition}

Note that this formula was discovered by Gauss \cite{G1965} for prime q, and so it is called Gauss's formula. By using the formula on $I_{q}(n)$, one can count the number of elements in $G_{q,n}$.
\begin{lemma}\label{lem:G} Let $n\in \N$ and $q$ be a prime power. Then $$ \lvert G_{q,n} \rvert =  q^n-n I_{q}(n)$$
\end{lemma}
\begin{proof}
It is clear that any root of an irreducible polynomial of degree $n$ over $\F_q$ can not be an element of its proper subfield. Hence the proof follows.
\end{proof}
\begin{example} Consider $\F_{q^n}$ for $n=105$. Then we have $d=1,3,5,7,15,21,35,105$ and by Lemma \ref{lem:G} we get 
$$| G_{q,n} | = q^{35}+q^{21}+q^{15}-q^7-q^5-q^3+q.$$
\end{example}

Now we give the definition of a norm and trace of an element in a finite field, see \cite[Chapter 2]{LN1997} for their properties.

\begin{definition}
For $\alpha \in \F_{q^n}$ the \textit{norm} $\Norm_{F_{q^n}/F_q}(\alpha)$ of $\alpha$ is defined by
$$\Norm_{F_{q^n}/F_q}(\alpha) = \alpha \cdot \alpha^q\cdot\alpha^{q^2}\cdots\alpha^{q^{n-1}}=\alpha^{(q^n-1)/(q-1)},$$
and the \textit{trace} $\Tr_{F_{q^n}/F_q}(\alpha)$ of $\alpha$  is defined by $$\Tr_{F_{q^n}/F_q}(\alpha)=\alpha+\alpha^{q}+\cdots+\alpha^{q^{n-1}}.$$
\end{definition}
In particular $\Norm_{F_{q^n}/F_q}(\alpha)$ and $\Tr_{F_{q^n}/F_q}(\alpha)$ are elements of $\F_{q}$.

\begin{definition}\cite[Chapter 5]{LN1997} Let $\chi$ be an additive and $\psi$ a multiplicative character of $\F_{q}$. Then $\chi$ and $\psi$ can be \textit{lifted} to $\F_{q^n}$ by setting $\chi'=\chi(\tau_{\F_{q^n}/\F_q}(\beta))$ for $\beta \in \F_{q^n}$ and $\psi'=\psi(N_{\F_{q^n}/\F_q}(\beta))$ for $\beta \in \F_{q^n}^{*}$. Also from the additivity of the trace and multiplicativity of the norm $\chi'$ is an additive and $\psi^{'}$ is a multiplicative character of $\F_{q^n}$.
\end{definition} 
By the definition of lifted character Gyarmati gave the following corollary in her paper \cite{G2015}. We give this corollary to use in the proof of Theorem \ref{thm:main}.
\begin{corollary}\cite[Corollary 2.1.]{G2015} Let $p>2$ be a prime number and $\left(\frac{.}{p}\right)$ be the Legendre symbol. Let $\gamma$ be the quadratic character of $\F_{p^n}$. Then for $\alpha \in \F^{*}_{p^n}$  $$\gamma(\alpha)=\left(\frac{\Norm_{F_{p^n}/F_p}(\alpha)}{p}\right)$$
\end{corollary}
In the following we define two new polynomials for a given polynomial over a finite field.
\begin{definition}
For $f(x)=a_kx^k+a_{k-1}x^{k-1}+\cdots+a_0 \in \F_{q^n}[x]$,  we define 
$$\tau_s(f)(x):=a_k^{q^s}x^k+a_{k-1}^{q^s}x^{k-1}+\cdots+a_0^{q^s} \in \F_{q^n}[x]$$ for  $0\le s \le n-1$ 
and
$$N_{F_{q^n}/F_q}(f):=\tau_0(f).\tau_1(f).\tau_2(f)\cdots\tau_{n-1}(f) \in \F_{q^n}[x].$$
\end{definition}
Next, we give a result from the book of Lidl and Niederreiter \cite{LN1997}, which will be the basis of the proof of our main theorem.
\begin{lemma}\cite[Exercise 5.64]{LN1997}\label{lem:LN} Let $i_1,\dots,i_j$ be distinct elements of $\F_{p^k}$, p odd, and  $\epsilon_1,\dots,\epsilon_j \in \{-1,+1\}$. Let $N(\epsilon_1,\dots,\epsilon_j)$ denote the number of $\alpha \in \F_{p^k}$ with
$$\gamma(\alpha+ i_s)=\epsilon_s\mbox{ for }s=1,2,\dots,j$$
where $\gamma$ is the quadratic character of $\F_{p^k}$. Then,
$$| N(\epsilon_1,\dots,\epsilon_j)-\frac{p^k}{2^j} | \le \bigg(\frac{j-2}{2}+\frac{1}{2^j}\bigg)p^{k/2} + \frac{j}{2}.$$
\end{lemma}

\section{Previous Results}
\label{sec:previous}
In this paper we improve bound given by Gyarmati \cite{G2015} on family complexity of Legendre sequences generated by irreducible polynomials. We will give construction method and result given by Gyarmati in this section.
We begin with the definition of well known Legendre sequence \cite{GMS2004,AMS1997}.
\begin{construction} \label{construction}
Let $K\ge 1$ be an integer and $p$ be a prime number. If $f \in \F_p[x]$ is a polynomial with degree $1\le k\le K$ and has no multiple zeros in $\F_p$, then define the binary sequence $E_p(f)=E_p=(e_1,\dots,e_p)$ by

\begin{displaymath}\label{1}
e_n = \left\{ \begin{array}{ll}
{\left(\frac{f(n)}{p}\right)} & \textrm{ for $(f(n),p)=1$}\\
+1 & \textrm{ for $p | f(n)$}
\end{array} \right.
\end{displaymath} Let $\sF(K,p)$ denote the set of all sequences obtained in this way. 
\end{construction}
Hoffstein and Lieman \cite{HF2001} presented the use of the polynomials $f$ given in Construction \ref{construction} but they did not give a proof for its pseudorandom properties. Goubin, Mauduit and Sárközy \cite{GMS2004} proved that the sequences obtained in this way have strong pseudorandom properties.

We now give the definition of the $f$-$complexity$ of a family $\sF$, which was first defined by Ahlswede et.~al.~\cite{AKMS2003} in 2003.
\begin{definition} 
The family complexity (or briefly \textit{$f$-complexity}) $C(\sF)$ of a family $\sF$ of binary sequences $E_N \in \{-1,+1\}^N$ of length $N$ 
is the greatest integer $j \geq 0$ such that for any $1 \leq i_1 < i_2< \cdots < i_j \leq N$ and any $\epsilon_1,\epsilon_2, \ldots,  \epsilon_j \in \{-1,+1\}$ 
there is a sequence $E_N = \{e_1,e_2,\ldots , e_N\}\in \sF$ with $$e_{i_1}=\epsilon_1,e_{i_2}=\epsilon_2, \ldots ,e_{i_j}=\epsilon_j.$$
The $f$-$complexity$ of a family $\sF$ is denoted by $\Gamma(F)$. 
\end{definition}
We note that the trivial upper bound on family complexity $\Gamma(\sF)$ in terms of family size $|\sF|$ is
\be \label{eqn.bound_f} \nn
2^{\Gamma(\sF)} \leq |\sF|.
\ee

We set the family of Legendre sequences generated by irreducible polynomials of degree $k$ over a prime field $\F_p$ by $\sF_{irred}(k,p)$:
$$\sF_{irred}(k,p):=\{E_p(f) : f \in \F_p\mbox{ monic irreducible polynomial with degree k}\}$$
This family has been studied for different measures (crosscorrelation etc.) in several papers \cite{GMS2004,G20092,GMS2014}.

Gyarmati \cite{G2015} recently proved a lower bound on the $f$-complexity of the family $\sF_{irred}(k,p)$, which says that the $f$-complexity is at least of order $\frac{p^{1/4}}{20\log{2}}$.

\begin{theoremA} \nn \cite{G2015} \label{thm:gyarmati} Let p be an odd prime and k be a positive integer. Define $c=\frac{1}{2}$ if $k\le\frac{p^{1/4}}{10\log p}$ and $c=\frac{5}{2}$ if $k>\frac{p^{1/4}}{10\log p}$ then
\begin{center} $\Gamma(\sF_{irred}(k,p)) \ge min\{p, \frac{k-c}{2\log 2}\log p\}.$
\end{center}
\end{theoremA}

In the next section,  we improve the lower bound given in Theorem A by using the formula $\lvert G_{p,k} \rvert$ given in Lemma \ref{lem:G} and Lambert $W$ function given in Definition \ref{def:W}.

\section{Main Method}
 \label{sec:contribution}

The main contribution of this paper is given in the following theorem, which is a new bound on the family complexity of Legendre sequences generated by irreducible polynomials. This new bound improves the bound given by Gyarmati \cite{G2015}. The comparison of two bounds is given in the next section.

\begin{theorem} \label{thm:main} Let $p$ be an odd prime and $k$ be a positive integer. Let  $A$ and $B$ be defined as
$$A=\frac{2p^{k/2}-2}{1+p^{-k/2}} \mbox{ and }  B=\frac{2 | G_{p,k}| p^{-k/2}-2}{1+p^{-k/2}}.$$ Then
$$\Gamma(\sF_{irred}(k,p)) \ge \log_{2}\bigg(\frac{A}{W(2^BA)}\bigg)$$
\end{theorem}

Before proving the theorem, we will give two auxiliary lemmas. In the first lemma, the solution of a logarithmic equation is obtained by Lambert $W$ function. In the second lemma, we give an upper bound on $j$ such that $| G_{p,k}| < N(\epsilon_1,\dots,\epsilon_j)$.
\begin{lemma} \label{lem:W} Let $A,B \in \R$. If
$Bx+ x\log_{2}x - A = 0$, then $x = \dd\frac{A}{W(A2^B)}$
\end{lemma}

\begin{proof} We have
$$x(B+ \log_{2}x) = A$$ or equivalently,
$$2^Bx(B+ \log_{2}x) = 2^BA.$$ Then we get
$$2^Bx(\log_{2}2^B + \log_{2}x) = 2^BA$$ and
$$2^Bx(\log_2(2^Bx)) = 2^BA.$$
Thus by Definition \ref{def:W} we have
$$ 2^Bx = \frac{2^BA}{W(2^BA)},$$ that is
$$ x = \frac{A}{W(2^BA)}.$$
\end{proof}

\begin{lemma} \label{lem:CY} Let $p$ be an odd prime and $k$ be a positive integer. Let $|G_{p,k}|$ be defined as in Lemma \ref{lem:G}. Let  $A$ and $B$ be defined as
$$A=\frac{2p^{k/2}-2}{1+p^{-k/2}} \mbox{ and }  B=\frac{2 | G_{p,k}| p^{-k/2}-2}{1+p^{-k/2}}.$$ Let $j$ be an integer such that $j < log_2\bigg(\frac{A}{W(2^BA)}\bigg)$. Let $\epsilon_1,\dots,\epsilon_j \in \{-1,+1\}$ and  $N(\epsilon_1,\dots,\epsilon_j)$ be defined as in Lemma \ref{lem:LN}. Then
$$| G_{p,k}| < N(\epsilon_1,\dots,\epsilon_j).$$
\end{lemma}
\begin{proof}
Assume that $|G_{p,k}| \ge N(\epsilon_1,\dots,\epsilon_j).$
Then by Lemma \ref{lem:LN}
$$|G_{p,k}| \ge \frac{p^k}{2^j}-p^{k/2}\bigg(\frac{1}{2^j}+\frac{(j-2)}{2}\bigg)-\frac{j}{2}.$$
Divide both sides by $p^{k/2}$
$$|G_{p,k}| p^{-k/2}\ge \frac{p^{k/2}}{2^j}-\bigg(\frac{1}{2^j}+\frac{(j-2)}{2}\bigg)-\frac{jp^{-k/2}}{2}.$$
Multiply both sides by $2(2^j)$, and so get the following equation array
$$2(2^j)|G_{p,k}| p^{-k/2}\ge 2p^{k/2}-2-2^j(j-2)-2^j jp^{-k/2}$$
$$2(2^j) |G_{p,k}| p^{-k/2}-2(2^j)+2^jj+2^jjp^{-k/2}\ge(2p^{k/2}-2)$$
$${(|G_{p,k}| p^{-k/2}-2)}2^j+2^jj(1+p^{-k/2})\ge (2p^{k/2}-2).$$
Divide both sides by $(1+p^{-k/2})$,
$$\frac{(2 |G_{p,k}| p^{-k/2}-2)}{(1+p^{-k/2})}2^j +2^jj \ge \frac{(2p^{k/2}-2)}{(1+p^{-k/2})}.$$
By definition of A and B, we have
$$B2^j +2^jj \ge A.$$
Hence, by Lemma \ref{lem:W} we obtain that
$$2^j \ge \frac{A}{W(2^BA)} \mbox{ or equivalently } j\ge log_2\bigg(\frac{A}{W(2^BA)}\bigg)
,$$
which is a contradiction.
\end{proof}

\begin{proofA}
We need to show the existence of $g \in \F_p[x]$ irreducible polynomial of degree $k$ such that 
$$\left(\frac{g(i_s)}{p}\right)=\epsilon_s\mbox{ for }s=1,2,\dots,j$$
for any tuple $(\epsilon_1,\epsilon_2,\dots,\epsilon_j)\in\{-1,+1\}^j$ and for any integer $j < log_2\left(\frac{A}{W(2^BA)}\right)$.
By Lemma \ref{lem:CY} we know that
$$| G_{p,k}| < N(\epsilon_1,\dots,\epsilon_j).$$
From definition of $N(\epsilon_1,\dots,\epsilon_j)$ we get that there exists $\alpha \in \F_{p^k} \backslash G_{p,k}$ such that

\begin{equation}\label{equ2}\gamma(\alpha+i_s)=\epsilon_s\mbox{ for }s= 1,2,\dots,j.
\end{equation}
Let $f(x)=x+\alpha \in \F_{p^k}[x]$ and we define $g(x):=N_{F_{p^k}/F_p}(f(x)) \in \F_p[x]$. We note that $g$ is an irreducible polynomial by using \cite[Lemma 2.4]{G2015}. 
We know that if p is a prime number, $\left(\frac{.}{p}\right)$ is the Legendre symbol and $\gamma$ is the quadratic character of $\F_{p^k}$ then for $\alpha \in \F^{*}_{p^k}$ we have
$$\gamma(\alpha) = \bigg(\frac{N_{F_{p^k}/F_p}(\alpha)}{p}\bigg).$$
By \cite[Lemma 2.3]{G2015}, we know that if $f\in \F_{p^k}[x]$ then for $\alpha \in \F_p$ we have  $$N_{F_{p^k}/F_p}(f(\alpha))=N_{F_{p^k}/F_p}(f)(\alpha).$$
Finally, using these and (\ref{equ2}) we get
\begin{eqnarray*}
\epsilon_s &=&\gamma(\alpha+i_s)=\gamma(f(i_s))=\left(\frac{N_ {F_{p^k}/F_p}(f(i_s))}{p}\right)=\left(\frac{N_{F_{p^k}/F_p}(f)(i_s))}{p}\right) \\
&=&\left(\frac{g(i_s)}{p}\right)\mbox{ for } s=1,2,\dots,j,
\end{eqnarray*} as desired.
\end{proofA}

Since $\sF(K,p) \supset \sF_{irred}(K,p)$ thus we can give the following corollary.
\begin{corollary}Let p be an odd prime and K be a positive integer. Let A and B be defined as in Theorem \ref{thm:main}. Then
$$\Gamma(\sF(K,p)) \ge \log_{2}\bigg(\frac{A}{W(2^BA)}\bigg)$$
\end{corollary}

\section{Comparison} \label{sec:compare}

We compare the bounds given in Theorem \ref{thm:main} on family complexity of Construction \ref{construction} and given in Theorem A \cite{G2015} with respect to the value of lower bounds and time to compute them.  

Firstly, it is seen in Figures 1 and 2 that the lower bound given in Theorem \ref{thm:main} is better than bound given by Gyarmati \cite{G2015}. Here the red lines show the bound in \cite{G2015} (see also Theorem A in this paper) and the blue lines show the bound given Theorem \ref{thm:main} in this paper.
\begin{figure}[ht]
	\includegraphics[width=8cm, height=6cm]{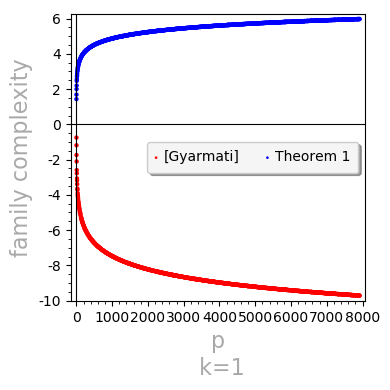}
	\includegraphics[width=8cm, height=6cm]{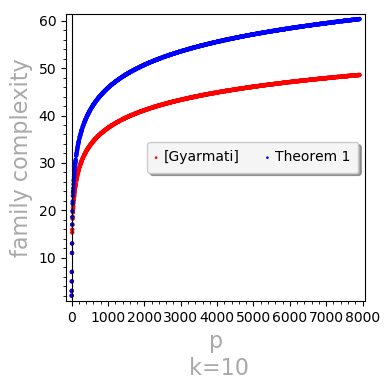}
	\caption{Lower bound on family complexity of Legendre sequence with respect to $p$ for fixed $k=1$ and $k=10$ respectively.}
	\label{fig:k:1-10}
\end{figure}
\begin{figure}
	\includegraphics[width=8cm, height=6cm]{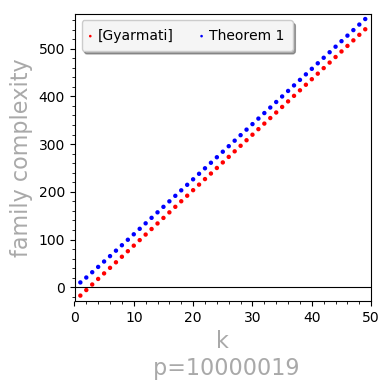}
	\includegraphics[width=8cm, height=6cm]{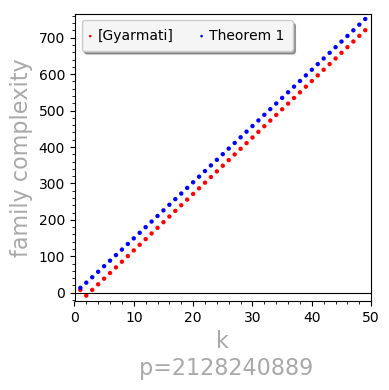}
	\caption{Lower bound on family complexity of Legendre sequence with respect to $k$ for fixed $p=10000019$ and $p=2128240847$ respectively.}
	\label{fig:k:1-10}
\end{figure}

In Figure 1, both bounds on family complexity of Construction \ref{construction} is plotted with respect to primes $p<8000$ for fixed $k=1$ and $k=10$ respectively. For $k=1$, it is seen that bound given by Gyarmati is negative, on the other hand, the bound in Theorem \ref{thm:main} is always positive. We note that bound given by Gyarmati turns into positive for $p \ge 2128240847$. For $k=10$, it is seen that both lower bounds are positive and the lower bound given in Theorem \ref{thm:main} is better than bound given by Gyarmati for all $p<8000$. In Figure 2 the lower bound on family complexity of Construction \ref{construction} is plotted in range $k \in [1,50]$ for fixed $p=10000019$ and $p=2128240847$ respectively. Here, $p=10000019$ is the first prime greater than $10^7$ and $p=2128240847$ is the first prime Gyarmati's bound turns into positive for $k=1$. In both cases, lower bounds are near to each other, but the lower bound in Theorem \ref{thm:main} is better.

\begin{figure}[hbt!]
	\includegraphics[width=8cm, height=6cm]{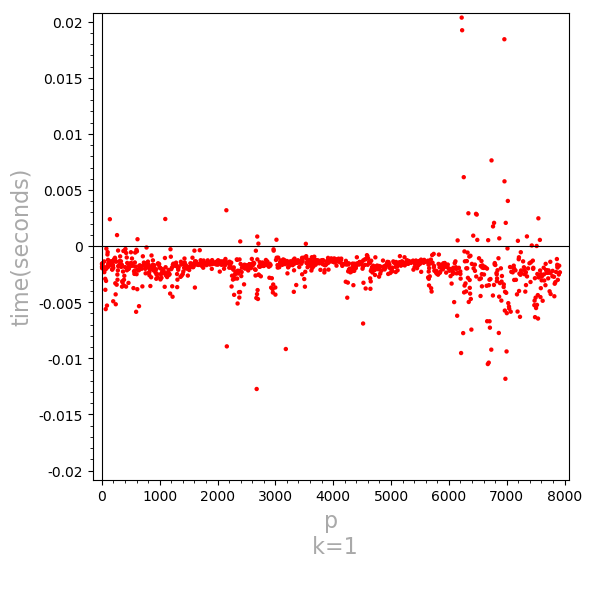}
	\includegraphics[width=8cm, height=6cm]{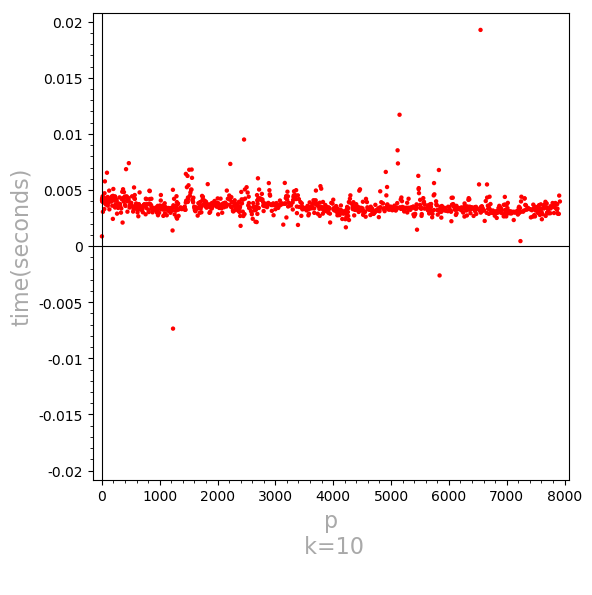}
	\caption{Times for calculate the family complexity of Legendre sequence with respect to $p$ for fixed $k=1$ and $k=10$ respectively.}
	\label{fig:k:1-10}
\end{figure}
\begin{figure}[hbt!]
	\includegraphics[width=8cm, height=6cm]{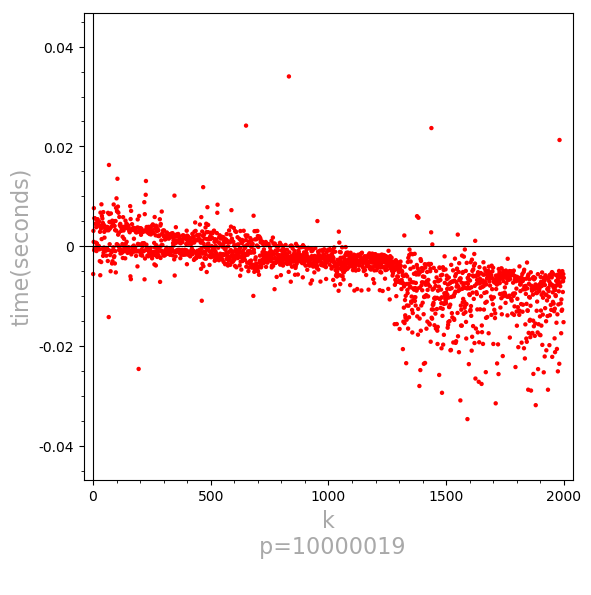}
	\includegraphics[width=8cm, height=6cm]{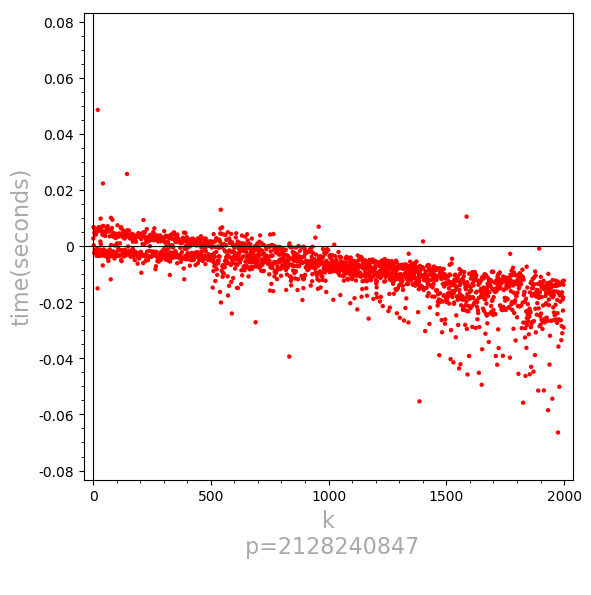}
	\caption{Times for calculate the family complexity of Legendre sequence with respect to $k$ for fixed $p=10000019$ and $p=2128240847$ respectively.}
	\label{fig:k:1-10}
\end{figure}

Secondly, we compare two bounds in terms of time complexity. We plotted in Figures 3 and 4 the difference between elapsed times for evaluating bounds given by Gyarmati \cite{G2015} and Theorem 1. In other words, we measured the elapsed times (in seconds) for calculating two bounds for all values of $p$ and $k$ that we have already examined in Figures 1 and 2. Then we plotted each difference of elapsed times of two bounds in Figures 3 and 4. It is seen that time to calculate two bounds are quite close to each other for all $p$ and $k$. For instance, in Figure 3 for $k=1$, time needed to calculate our bound is more than bound given by Gyarmati, the difference is at most 0.005 seconds. On the other hand, for $k=10$ bound given by Gyarmati is calculated slower and the difference between elapsed times is at most 0.01 seconds. Similarly, in Figure 4 it is seen that time to calculate two bounds differ from each other at most 0.06 seconds for primes $p=10000019$, $p=2128240847$ and $k \in \{1,2,\ldots,2000\}$. We conclude that the bound given in Theorem 1 can be calculated very fast for arbitrarily large prime powers and it only differs a few milliseconds form evaluating a bound depending only on $p$ and $k$. 

\section*{Acknowledgment}

	The authors are supported by the Scientific and Technological Research	Council of Turkey (TÜBİTAK) under Project No: \mbox{116R026}.

	\bibliographystyle{plain}      
	\bibliography{arxiv}   
	
\end{document}